%% file: art.tex
\documentclass[11pt, twoside]{amsart}

\usepackage{amsmath}
\usepackage{amsfonts}
\usepackage{amssymb}
\usepackage{amsthm}
\usepackage{bbm}
\usepackage{verbatim} 
\usepackage{enumitem}

\usepackage[T1]{fontenc}

\theoremstyle{plain}
\newtheorem{thm}{Theorem}[section]
\newtheorem{cor}[thm]{Corollary}
\newtheorem{lma}[thm]{Lemma}
\newtheorem{prp}[thm]{Proposition}

\newtheorem{fact}[thm]{Fact}

\theoremstyle{definition}
\newtheorem{defi}[thm]{Definition}
\newtheorem{rmk}[thm]{Remark}
\newtheorem{exa}[thm]{Example}
\newtheorem{quest}[thm]{Question}

\theoremstyle{remark}

\numberwithin{equation}{section}

\newcommand{\reduct}{\upharpoonright}

\newcommand{\mcA}{\mathcal{A}}
\newcommand{\mcB}{\mathcal{B}}
\newcommand{\mcC}{\mathcal{C}}
\newcommand{\mcD}{\mathcal{D}}

\newcommand{\mcH}{\mathcal{H}}

\newcommand{\mcM}{\mathcal{M}}
\newcommand{\mcN}{\mathcal{N}}

\newcommand{\mbK}{\mathbf{K}}

\newcommand{\mbbN}{\mathbb{N}}

\newcommand{\mbbQ}{\mathbb{Q}}

\newcommand{\FR}{Fra\"iss\'e }
\newcommand{\Fr}{Fra\"iss\'e}

\title{Homogenizable structures and model completeness}
\author{Ove Ahlman}
\thanks{E-mail: ove@math.uu.se}
\address{Ove Ahlman, Department of Mathematics, Uppsala University, Box 480, 75 106 Uppsala, Sweden}
\email{ove@math.uu.se}

\begin{document}
\maketitle
\begin{abstract}
A homogenizable structure $\mcM$ is a structure where we may add a finite number of new relational symbols to represent some $\emptyset-$defi\-nable relations in order to make the structure homogeneous. In this article we will divide the homogenizable structures into different classes which categorize many known examples and show what makes each class important. We will show that model completeness is vital for the relation between a structure and the amalgamation bases of its age and give a necessary and sufficient condition for an $\omega-$categorical model-complete structure to be homogenizable.
\end{abstract}
\section{Introduction}
\input{introduction}
\section{Homogenizable structures}\label{prelim}
\input{preliminaries}
\section{Boundedly homogenizable structures}\label{symsec}
\input{sym}
\section{Uniformly homogenizable structures}\label{unisec}
\input{uniform}
\section{Unavoidably homogenizable structures}\label{unavoidsec}
\input{binclassi}

\section{Unary homogenizable structures}\label{unsec}
\input{one}

\end{document}

%% file: introduction.tex
A structure $\mcM$ is called \textbf{homogeneous} (sometimes called ultrahomogeneous \cite{H}) if for each $\mcA\subseteq \mcM$ and embedding $f:\mcA\rightarrow \mcM$, $f$ may be extended into an automorphism of $\mcM$ i.e. there is an isomorphism $g:\mcM\rightarrow\mcM$ such that $g\reduct A = f$. A structure over a finite relational language is homogenizable if we can add new relational symbols to the structure's signature representing a finite number of formulas, such that the new expanded structure is homogeneous (see Definition \ref{homidef} for details). The homogenizable structures are found in a variety of areas of mathematics, especially when studying random structures or structures with some excluded subgraphs, also called $\mcH-$free structures \cite{AK2, AK, BBPP, C2, KPR, M}. In 1953 \FR \cite{Fr} studied homogeneous structures and found that for each set of finite structures $\mbK$ satisfying the properties HP, JEP and AP there is a unique infinite countable homogeneous structure $\mcM$ such that $\mbK$ is exactly the set of finite substructures of $\mcM$ (up to isomorphism). Covington \cite{C} extended \Fr 's result to sets $\mbK$ which instead of AP satisfy the so called ``local failure of amalgamation'' property, and concluded that each of these sets induces a unique homogenizable structure which is model-complete. This study of the homogenizable structures gives a sufficient yet not necessary condition for a set of structures to generate a homogenizable structure. In a more recent study Hartman, Hubi$\check{\text{c}}$ka and Ne$\check{\text{s}}$et$\check{\text{r}}$il \cite{HHN} explores the concept of homogenizable structures by investigating how high an arity is needed among the newly added relational symbols and call this number the relational complexity. The article shows that if $\mbK$ is a set of structures which are restricted by a finite minimal family of finite connected relational structures then $\mbK$ generates a homogenizable structure. This is a sufficient, but not necessary condition for a set of finite structures to induce an infinite homogenizable structure.
In even more generality relational complexity has been studied by Cherlin \cite{Ch} among others, who focus on properties of the automorphism group. The concept of relational complexity and the results in the current article are easy to merge, as we work closely to the homogenizing formulas. However, the question whether all the structures studied by \cite{Ch,C,HHN} are boundedly homogenizable (see Definition \ref{symdef}) or not remains open. \\\indent
In this paper we use a finite relational vocabulary and study countably infinite homogenizable structures, what the formulas which homogenize them look like, how their set of finite substructures behave and how the types of the structure affect the homogenization. In section \ref{prelim} we introduce the subject and give some basic definitions, but we will also provide many instructive examples pointing out how different kinds of homogenizable structures relate to each other. The main result is the following theorem which gives a necessary and sufficient condition for $\omega-$categorical model-complete structures to be homogenizable (see Definition \ref{SEAPdef} for the meaning of SEAP). 
\begin{thm}\label{mcmainthm}Let $\mcM$ be a countably infinite structure which is model-complete and $\omega-$categorical. $Age(\mcM)$ satisfies SEAP if and only if $\mcM$ is homogenizable.
\end{thm}
\indent Section \ref{symsec} studies the \textit{boundedly homogenizable} structures. We prove that they are model-complete and hence conclude with the following theorem, which is an interesting extension of \Fr 's theorem.
\begin{thm}\label{symmainthm}
Let $\mbK$ be a set of structures closed under isomorphism and satisfying HP and AP. Then there is a unique countably infinite structure $\mcM$ such that $Age(\mcM) =\mbK$ and $\mcM$ is boundedly homogenizable.
\end{thm}
\noindent In other words, the theorem states that the unique homogeneous structure having age equal to $K$, also called the \Fr-limit, is the unique boundedly homogeneous structure.
\\\indent In section \ref{unisec} we study the \textit{uniformly homogenizable} structures, and prove that these are the structures where we may find a universal witness which witnesses all the homogenizing formulas. We will also see that the uniformly homogenizable structures contain many homogenizable structures which are ``easy'' to homogenize. In Section \ref{unavoidsec} we do a quick study of the \textit{unavoidably homogenizable} structures, the set of structures which are as close to being homogeneous as it gets. In all three of the sections \ref{symsec}, \ref{unisec} and \ref{unavoidsec} we prove that the homogenizable structures have certain conditions associated to the amalgamation bases of their ages and that we may extend certain self-embeddings into automorphisms. \\\indent
Unary relation symbols are often considered with special care, and so we devote Section \ref{unsec} to the study of the structures we may homogenize by only adding new unary relational symbols. The epicenter of this is Theorem \ref{unatouni} which connects unary boundedly homogenizable structures with the uniformly homogenizable structures.

%% file: preliminaries.tex
\noindent We will consider a \textbf{finite relational vocabulary} $V$ which is a finite set of relational symbols of finite arities, so in particular has no constant or function symbols. In this paper we will only consider first order formulas over such a vocabulary. The formulas which are of the form $\exists x_1\ldots\exists x_n\varphi$ where $\varphi$ is quantifier free are called \textbf{$\Sigma_1-$formulas}. We will denote $V-$structures by calligraphic letters $\mcA,\mcB,\mcM,\mcN,\ldots$ and their respective universes with roman letters $A,B,M,N,\ldots$. Ordered tuples $\bar a,\bar b,\bar x,\ldots$ may at times be (notationally) identified with the set of their elements. The meaning will be made obvious from what operations are applied. The set $\{1,\ldots,n\}$ may be written with the abbreviation $[n]$. If $\mcM$ is a structure and $A\subseteq M$, then $\mcM\reduct A$ is the substructure of $\mcM$ with universe $A$. If $V\subseteq V'$ are both vocabularies and $\mcM$ is a $V'-$structure, then the \textbf{reduct} of $\mcM$ to $V$, written $\mcM\reduct V$ is the $V-$structure which we get when we remove all relations in $V'-V$ from $\mcM$. If $f: A\rightarrow B$ is a function and $C\subseteq A$ then $f\reduct C$ is the function $f$ restricted to the domain $C$. Although we use $\reduct$ for many things, the context should always make the intention clear. If $\bar a\in M$ then $tp^\mcM(\bar a/\bar b)$ is the set of all formulas (with parameters from $\bar b$) which $\bar a$ satisfies, also called the complete type of $\bar a$ over $\bar b$. If $\varphi(\bar x)\in tp(\bar a)$ is such that for every formula $\psi\in tp(\bar a)$, $\mcM\models \forall \bar x(\varphi(\bar x)\rightarrow \psi(\bar x))$ we say that $\varphi$ \textbf{isolates} $tp(\bar a)$. A model $\mcM$ is \textbf{model-complete} if $Th(\mcM)$ (the theory of all true sentences in $\mcM$) is such that every embedding between models of $Th(\mcM)$ is elementary. It is a known fact (\cite{H}, Theorem 8.3.1) that $\mcM$ is model-complete if and only if each formula is equivalent to a $\Sigma_1-$formula over $Th(\mcM)$.
\\\indent If $\mcM$ is a relational structure then $Age(\mcM)$ is the class of all finite structures which are embeddable in $\mcM$.
Let $\mbK$ be any set of finite structures. We say that $\mbK$ satisfies the \textbf{hereditary property}, written \textbf{HP}, if for each $\mcA\in\mbK$, if $\mcB\subseteq \mcA$ then $\mcB\in\mbK$. If, for each $\mcB,\mcC\in\mbK$, there exists a structure $\mcD\in\mbK$ in which both $\mcB$ and $\mcC$ are embeddable then $\mbK$ have the \textbf{joint embedding property}, written \textbf{JEP}.
A structure $\mcA$ is an \textbf{amalgamation base} for $\mbK$ (or just an amalgamation base if $\mbK$ is clear from the context), if for any structures $\mcB,\mcC\in \mbK$ and any embeddings $f:\mcA\rightarrow \mcB$, $g: \mcA\rightarrow \mcC$ there is a structure $\mcD\in\mbK$, called an \textbf{amalgam} for $f$ and $g$, and embeddings $f_0:\mcB\rightarrow \mcD, g_0:\mcC\rightarrow \mcD$ such that for each $x\in A$, $f_0(f(x)) = g_0(g(x))$. In the special case when $f_0, g_0$ can be chosen so that $f_0(B) \cap g_0(C) = g_0(A)=f_0(A)$ we call $\mcA$ a \textbf{disjoint amalgamation base}. If each $\mcA\in\mbK$ is an (disjoint) amalgamation base for $\mbK$ then $\mbK$ satisfies the (disjoint) amalgamation property, in short written AP. 
We note that for sets $\mbK$ containing only relational structures and satisfying HP and AP the property JEP follows, since the empty structure is an amalgamation base.
\begin{thm}[\FR \cite{Fr}]\label{FRthm} If $\mbK$ is a class of relational structures closed under isomorphism which satisfies HP and AP, then there is a unique countably infinite homogeneous structure $\mcM$ such that $Age(\mcM)=\mbK$. The structure $\mcM$ is called the \FR limit of $\mbK$.
\end{thm}
\noindent Following the concept of being homogeneous we will in this article study structures which are so close to homogeneous that it is only a matter of adding finitely many symbols to already existing definable relations. Recall from the beginning of this section that we only consider finite vocabularies.
\begin{defi}\label{homidef}A $V-$structure $\mcM$ is \textbf{homogenizable} if there exists a finite amount of formulas $\varphi_1(\bar x_0),\ldots,\varphi_n(\bar x_n)$, called the homogenizing formulas, such that if we, for each $i\in\{1,\ldots,n\}$, create a new relational symbol $R_i$ of the same arity as $\varphi_i$ and put $V' = V\cup\{R_1,\ldots,R_n\}$, then there is a homogeneous $V'-$structure $\mcN$ such that $\mcN\reduct V= \mcM$ and for each $\bar a\in N$ and $i\in\{1,\ldots,n\}$ $\mcN\models R_i(\bar a) \leftrightarrow \varphi_i(\bar a)$. If all homogenizing formulas are $\Sigma_1$, then we say that $\mcM$ is \textbf{$\Sigma_1-$homogenizable}. A homogenizable structure is \textbf{unary homogenizable} if all homogenizing formulas have only one free variable.
\end{defi}
A structure $\mcM$ is called $\omega-$\textbf{categorical} if $Th(\mcM)$ has a single countable model up to isomorphism. The following well known fact about $\omega-$categorical structures will be used without mention throughout this article.
\begin{fact}If $\mcM$ is a structure then the following are equivalent.
\begin{itemize} 
\item $\mcM$ is $\omega-$categorical.
\item For each $n$ there exists only a finite number of $n-$types over $\emptyset$.
\item Each type over $\emptyset$ is isolated.
\end{itemize}
\end{fact}
Over a finite vocabulary it is clear that a structure which is homogenizable or homogeneous is also $\omega-$categorical. For an $\omega-$categorical structure $\mcM$ over a finite vocabulary, all types being isolated by quantifier free formulas (called quantifier elimination) is equivalent to $\mcM$ being homogeneous. Weakening the assumptions to $\mcM$ only being homogenizable it hence becomes natural to ask how the types now are being isolated. 
The amalgamation property still holds in homogenizable structures over realizations of types which are isolated by quantifier-free formulas.
The converse of the following lemma is not even true for $\Sigma_1-$homogenizable structures - see Example \ref{nonsymexa}.
\begin{lma}\label{prpamfree}
 If $\mcM$ is a structure and $\bar a \in M$ is such that $tp(\bar a)$ is isolated by a quantifier free formula then $\mcA = \mcM\reduct \bar a$ is an amalgamation base for $Age(\mcM)$.
\end{lma}
\begin{proof}
Assume that $f:\mcA\rightarrow \mcB$ and $g:\mcA\rightarrow \mcC$ for some $\mcB,\mcC\in Age(\mcM)$. As $tp^\mcM(\bar a)$ is isolated by the atomic diagram $\chi_\mcA$ of $\mcA$ we see that $tp^\mcM(\bar a)=tp^\mcM(f(\bar a))=tp^\mcM(g(\bar a))$. Hence $\mcM\models \chi_{\mcA}(\bar a)\rightarrow \exists \bar y \chi_\mcB (\bar a,\bar y)$ and $\mcM\models \chi_{\mcA}(\bar a)\rightarrow \exists \bar y \chi_\mcC (\bar a,\bar y)$. Let $\bar b$ and $\bar c$ be such that $\mcM\models \chi_{\mcA}(\bar a)\rightarrow  \chi_\mcB (\bar a,\bar b) \wedge \chi_\mcC (\bar a,\bar c)$ and put $\mcD = \mcM\reduct \bar a\bar b\bar c$, so $\mcD\in Age(\mcM)$. If $h_0:\mcB\rightarrow\mcD$ according to $\chi_\mcB$ and $h_1:\mcC\rightarrow\mcD$ according to $\chi_\mcC$, then $h_0(f(x)) = h_1(g(x))$ for each $x\in A$. We conclude that $\mcD$ is an amalgam for $f$ and $g$, thus $\mcA$ is an amalgamation base.
\end{proof}
Adding some more assumptions we may prove the converse of the previous lemma.
\begin{lma}\label{amtofree} Let $\mcM$ be $\omega-$categorical, model-complete and for $\bar a\in M$ let $\mcA=\mcM\reduct \bar a$. If $\mcA$ is an amalgamation base for $Age(\mcM)$ then $tp(\bar a)$
is isolated by a quantifier free formula. 
\end{lma}
\begin{proof}
If $\bar a'$ has the same atomic diagram as $\bar a$ let $\varphi,\psi$ be the $\Sigma_1-$formulas isolating the types of each respective tuple. Let $\bar b,\bar c$ be tuples witnessing the existential quantifiers isolating formulas of $\bar a$ respectively $\bar a'$ and put $\mcB=\mcM\reduct \bar b\bar a$ and $\mcC = \mcM\reduct \bar c\bar a'$. Since $\mcA$ is an amalgamation base the embeddings $f:\mcA\rightarrow \mcB$, $g:\mcA\rightarrow \mcC$ should have an amalgam $\mcD\subseteq \mcM$ with embeddings $f_0:\mcB\rightarrow \mcD$ and $g_0:\mcC\rightarrow \mcD$. However the atomic diagram of $f_0(\bar a\bar b)$ and $g_0(\bar c\bar a')$ implies that $tp(f_0(f(\bar a)))=tp(\bar a)$ and $tp(g_0(g(\bar a)))=tp(\bar a')$ respectively. As $\mcD$ is an amalgam of $f$ and $g$ it thus follows that $tp(\bar a)=tp(f_0(f(\bar a)))=tp(g_0(g(\bar a)))=tp(\bar a')$.
\end{proof}

\begin{lma}\label{qrfreelma}Let $\mcM$ be a saturated countably infinite structure with $\bar a\in M$ and put $\mcA = \mcM\reduct \bar a$. Each embedding $f:\mcA\rightarrow \mcM$ may be extended into an automorphism of $\mcM$ if and only if $tp(\bar a)$ is isolated by a quantifier free formula.
\end{lma}
\begin{proof} If $f:\mcA\rightarrow \mcM$ is an embedding then by the saturation of $\mcM$, $tp(\bar a) = tp(f(\bar a))$ if and only if $f$ may be extended into an automorphism. But $f$ is an embedding if and only if $\bar a$ and $f(\bar a)$ satisfies the same atomic diagram.
\end{proof}
\noindent The previous lemma hints that having a type isolated by a quantifier free formula implies that the specific tuple does its part in trying to make the structure homogeneous. Following from this we introduce three new concepts of homogenizable structures, assuming different levels of how easy it is to find a type which is isolated by a quantifier free formula.
\begin{defi}
Let $\mcM$ be a structure and $k\in\mbbN$. We say that $\mcM$ is $k-$\textbf{unavoidably homogenizable} if, for each $n \in \{k,k+1,\ldots\}$, each $n-$type is isolated by quantifier free formula. $\mcM$ is \textbf{unavoidably homogenizable} if it is $k-$unavoidably homogenizable for some $k\in\mbbN$.
\end{defi}
\noindent The unavoidably homogenizable structures are as close to being homogeneous as it gets, yet they do not seem easy to classify completely as we will see in Section \ref{unavoidsec}.
\begin{defi}
A homogenizable structure $\mcM$ is called \textbf{uniformly homogenizable} if there is a tuple $\bar a\in M$ such that for any $\bar b\in M$, $tp(\bar a\bar b)$ is isolated by a quantifier free formula.
\end{defi}
\noindent As will be made clear in Section \ref{unisec}, the uniformly homogenizable structures contain many trivial kinds of homogenizable structures yet are also quite central among homogenizable structures as Proposition \ref{mequni} and Theorem \ref{unatouni} show. 
\begin{defi}\label{symdef} A homogenizable structure $\mcM$ is \textbf{boundedly homogenizable} if for each $\bar a\in M$ there exists a $\bar b$ such that $tp(\bar a\bar b)$ is isolated by a quantifier free formula.
\end{defi}
\noindent The boundedly homogenizable structures form a very broad class of structures and it seems like most examples found in the literature fall in here, as we see in the examples bellow.
\begin{rmk} We note the following implications. Examples showing that these are all strict are provided below and in later sections we will explore some of the classes more. The second implication follows using Lemma \ref{unavoidlma} while the fourth implication uses Lemma \ref{symtosigma}.
\[\text{ Homogeneous} \Rightarrow \text{ Unavoidably homogenizable}
\Rightarrow\]\[ \text{ Uniformly homogenizable} \Rightarrow \text{Boundedly homogenizable} \Rightarrow  \]\[\text{$\Sigma_1-$homogenizable} \Rightarrow \text{Homogenizable}.\]
\end{rmk}

\begin{exa}[Kolaitis, Pr\"omel, Rothschild \cite{KPR}]\label{bipexa}
For some $l\in\mbbN$ let $\mbK_n$ be all $l-$partite graphs with universe $\{1,\ldots,n\}$, edge relation $E$ and let $\mu_n$ be the probability measure on $\mbK_n$ such that for each $\mcM\in\mbK_n$, $\mu_n(\mcM) = \frac{1}{|\mbK_n|}$. Put $T_\mbK$ to be the theory (called the almost sure theory) consisting of all sentences $\varphi$ such that 
\[\lim_{n\rightarrow\infty}\mu_n(\{\mcM\in \mbK_n : \mcM\models \varphi \}) = 1\]
$T_\mbK$ is $\omega-$categorical and the unique countable model $\mcN\models T_\mbK$, called the random $l-$partite graph has following property: For each $a,b\in N$, $a$ and $b$ belong to the same part if and only if $\mcN$ satisfies
\[\exists x_{2}\ldots\exists x_l \bigwedge_{i=2}^l\bigwedge_{i\neq j}(aEx_i \wedge bEx_i \wedge x_iEx_j).\]
If we let $\xi(a,b)$ be the formula above, then it is easy to prove $\xi$ is a homogenizing formula, thus $\mcN$ is homogenizable.
Using a generalization of $\xi$ we may, for any tuple $\bar a\in N$, find $l$ elements $b_1,\ldots ,b_l\in N$ such that the tuple $\bar ab_1\ldots b_l$ is a connected graph and of diameter 3 in $\mcN$. It is easy to see that any such tuple $\bar a b_1\ldots b_l$ in $\mcN$ has a type which is isolated by a quantifier free formula, and hence we have found that $\mcN$ is boundedly homogenizable. The structure is not uniformly homogenizable since for any tuple $\bar b$ we can find an element $c$ which is not adjacent to any elements in $\bar b$, which clearly means that the tuple $c\bar b$ may be mapped such that $c$ is in the wrong part compared to the tuples in $\bar b$.
\end{exa}
\noindent As we will see in Section \ref{unisec} and especially Proposition \ref{finsatprp}, it is easy to create a uniformly homogenizable structure. We may just take the infinite complete graph and remove a single edge. The following example however shows that they may not be at all trivial even though the homogenization still is.
\begin{exa}\label{bipeqexa}
Let $\mcM$ be the random $l-$partite graph obtained from Example \ref{bipexa}, but where we add new elements $a_1,a_2$ to the universe, and add two new relations $P$ and $R$ to the vocabulary. Let $P$ be unary and $P^\mcM = \{a_1,a_2\}$. 
Let $R$ be a $3-$ary relation such that $\mcM\models R(b,c,a_1)$ if and only if $b$ and $c$ are in the same part. If $b$ and $c$ are not in the same part then $\mcM\models R(b,c,a_2)$.  This is the construction from Proposition \ref{mequni} and $\mcM$ is hence a uniformly homogenizable structure. For any tuple $\bar c$, the type $tp(\bar c a_1a_2)$ will be isolated by a quantifier free formula as $a_1,a_2$ will be able to point out which elements in $\bar c$ belong to the same part. $Age(\mcM)$ does not satisfy the local failure of amalgamation property (LFA) discussed by Covington \cite{C}. This follows quickly since the random $l-$partite graph does not satisfy LFA (Covington points this out for bipartite graphs, and a similar reasoning works for $l-$partite graphs) and the same argument can be extended to $Age(\mcM)$.
\end{exa} 
\noindent In the next example we see that the strict order property may appear and thus there are boundedly homogenizable non-homogeneous structures which are not simple (see \cite{TZ} for detailed definitions of these concepts).
\begin{exa}[Bodirsky et. al. \cite{BBPP}]\label{treeexa}Let $\mcM$ be the countable, binary do\-wn\-wards-branching, dense, unbounded, semi-linear order without joins. This structure is boundedly homogenizable with a single homogenizing formula $C(x,y,z)$ saying, for incomparable vertices $x$, $y$ and $z$ that there is an element $c$ which is larger than $x$ and $y$ but still incomparable with $z$, i.e. in some sense $x$ and $y$ are closer to each other than to $z$.
For any tuple $\bar b$, and triple $b_0,b_1,b_2\in \bar b$ such that $\mcM\models C(b_0,b_1,b_2)$ let $c_0$ be an element witnessing this and let $\bar c$ be a tuple containing such witnesses for any triple in $\bar b$ satisfying $C$.  If this process is continued for $\bar b\bar c$ we will, in a finite amount of steps, reach a tuple $\bar b\bar d$ which is a finite binary tree. Thus this tuple has a type which is isolated by a quantifier free formula.
\end{exa}
\noindent The boundedly homogenizable structures are very common in examples of homogenizable structures in the literature. There are however homogenizable structures which are not boundedly homogenizable. 

\begin{exa}\label{nonsymexa} Let $\mcM = (\mbbQ^+ \cup\{0\}, <)$ be  the countable dense linear ordering without upper bound but with a lower endpoint. This structure is not homogeneous since the smallest element may never be mapped by an automorphism to anything but itself. However the formula $\exists y (y<x)$ creates a $\Sigma_1-$homogenization for $\mcM$. No type in $\mcM$ is isolated by a quantifier free formula, since the least element in any tuple can not be determined (without quantifiers) to be the endpoint $0$ or not. Hence $\mcM$ is $\Sigma_1-$homogenizable but not boundedly homogenizable and not model-complete. We have that $Age(\mcM) =Age( (\mbbQ, <))$. 
\end{exa}
\noindent All examples up until now have been $\Sigma_1-$homogenizable. However there are non $\Sigma_1-$homogenizable structures, as the following example shows. This article does not further explore these structures, and the following question remains open.
\begin{quest}Does there exists a $\Sigma_n-$ but not $\Sigma_{n-1}-$homogenizable structure for each $n\in \mbbN$?
\end{quest}
\begin{exa}\label{nosigmaex}
Let $\mcM$ be the structure with universe $(\mbbQ^+\cup\{0\}) \dot\cup (\mbbQ^- \cup \{0\})$, and with a binary relation $<$ interpreted as the strict linear order on each part of the disjoint union, yet $<$ does not compare elements from different parts of the disjoint union.
This structure is homogenizable by the formulas $\exists y(x < y)$, $\exists y(y<x)$ and $\exists y\forall z (\neg z< y \wedge y< x)$. The first two formulas makes the two endpoints stand out and the third formula makes it impossible to mix together the elements of $\mbbQ^-$ and $\mbbQ^+$. It is thus clear that $\mcM$ is unary homogenizable. We may also notice that the structure $\mcM$ is not model-complete, since the structure with universe $\mbbQ\dot\cup \mbbQ$ together with the expected order relation on each of the two disjoint sets, has the same age and is homogeneous. \\\indent
Let $f:\mcM\rightarrow\mcM$ be such that $\mbbQ^+\cup\{0\}$ is mapped to the half-open interval $[-1,0)$ and $\mbbQ^-\cup\{0\}$ is mapped to $(0,1]$, both in an order-preserving way. This function is a self-embedding of $\mcM$ and hence it preserves the $\Sigma_1-$formulas. We may conclude that any element $a\in\mbbQ^-$ and $b\in\mbbQ^+$ satisfy the same $\Sigma_1-$formulas in $\mcM$. We conclude that, since $\mcM$ is unary homogenizable, it is not possible to homogenize $\mcM$ using only $\Sigma_1-$formulas.
\end{exa}
By Theorem \ref{FRthm} it is sufficient for a set of finite structures to satisfy HP and AP in order to generate a homogeneous structure, and one might ask if there is a similar condition which guarantees the existence of a homogenizable structure. The following property solves this problem for ages of $\omega-$categorical structures.
\begin{defi}\label{SEAPdef}
Let $\mbK$ be a class of finite structures and $k,m\in \mbbN$. Define the $(k,m)$-\textbf{subextension amalgamation failure property} (SEAP$_{k,m}$) to be the following. For any $\mcA,\mcB,\mcC\in \mbK$ with embeddings $f:\mcA\rightarrow\mcB,\>\>g:\mcA\rightarrow \mcC$ without an amalgam, there exist $\mcA_0\subseteq \mcA$, $\mcB_0\supseteq \mcB$ and $\mcC_0\supseteq \mcC$ with $|A_0|<k$, $|B_0|-|B|<m$ and $|C_0|-|C|<m$ such that $f_0:\mcA_0\rightarrow \mcB_0$ and $g_0:\mcA_0\rightarrow\mcC_0$ with $f_0=f\reduct {A_0}$ and $g_0=g\reduct {A_0}$ do not have an amalgam. We say that $\mbK$ satisfies SEAP if it satisfies SEAP$_{k,m}$ for some $k,m\in\mbbN$.
\end{defi}
It is clear that any set of structures which satisfies AP will satisfy SEAP since SEAP only speaks about how failing amalgamations should behave.
As we have all necessary definitions we may now start with the lemmas necessary to prove Theorem \ref{mcmainthm}. The proof of the first lemma is done by assigning relations on all small enough types and then showing, using SEAP, that this creates a homogeneous structure.
\begin{lma} Let $\mcN$ be a model-complete $\omega-$categorical countably infinite structure. If $Age(\mcN)$ satisfies SEAP then $\mcN$ is homogenizable.
\end{lma}
\begin{proof}
Let $m$ and $k$ be numbers such that $Age(\mcN)$ satisfies SEAP$_{k,m}$. As $\mcN$ is $\omega-$categorical there are only a finite amount of types of the tuples of size less than $k$. Let $V'\supseteq V$ be the extended vocabulary where, for each $i<k$, and $i-$type over $\emptyset$ there is an $i-$ary new relational symbol. Let $\overline \mcN$ be the $V'-$structure such that $\mcN = \overline \mcN \reduct V$ and for each relational symbol $R$ in $V'-V$ there is a distinct complete type $p(\bar x)$ over $\emptyset$ in $\mcN$ such that for each $\bar a \in N$, $\mcN\models p(\bar a)$ if and only if $\overline \mcN\models R(\bar a)$ and all interpretations of the relations $V'-V$ in $\overline\mcN$ are disjoint. Thus the new relational symbols isolate the $i-$types in $\mcN$ for each $i<k$ and as $\mcN$ is $\omega-$categorical these relations are $\emptyset-$definable. We claim that $\overline \mcN$ is homogeneous, and thus $\mcN$ was homogenizable.\\\indent
In search for a contradiction, assume that $\overline \mcN$ is not homogeneous, so there exist tuples $\bar a_1,\bar a_2\in \overline N$ with the same atomic diagram such that $tp^{\overline \mcN}(\bar a_1)\neq tp^{\overline \mcN}(\bar a_2)$. As $\overline{\mcN}$ is just an expansion by $\emptyset-$definable relations, it follows that $tp^\mcN(\bar a_1) \neq tp^\mcN(\bar a_2)$. The model-completeness and $\omega-$categoricity of $\mcN$ implies that all types are isolated by $\Sigma_1-$formulas. Let $\bar c\supseteq \bar a_1$ and $\bar b\supseteq \bar a_2$ be tuples such that the existential quantifiers of the formulas isolating $tp^\mcN(\bar a_1)$ and $tp^\mcN(\bar a_2)$ respectively are witnessed by some subtuple. Let $\mcA = \mcN\reduct \bar a_1, \mcB = \mcN\reduct \bar b, \mcC = \mcN \reduct \bar c$, and note that since $tp^\mcN(\bar a_1)\neq tp^\mcN(\bar a_2)$ is witnessed in $\mcB$ and $\mcC$, the functions $f:\mcA\rightarrow \mcB$ and $g:\mcA\rightarrow \mcC$, where $f$ maps $\bar a_1$ to $\bar a_1$ and $g$ maps $\bar a_1$ to $\bar a_2$, can not have an amalgam in $Age(\mcN)$. As $Age(\mcN)$ satisfies SEAP$_{k,m}$ there exists $\mcA_0\subseteq \mcA, \mcB_0\supseteq \mcB$ and $\mcC_0\supseteq \mcC$ such that $|A_0|<k$, $|\mcB_0|-|\mcB|<m, |\mcC_0|-|\mcC|<m$ and the induced functions $f\reduct A_0$ and $g\reduct A_0$ do not have an amalgam. This in turn implies that there are embeddings $f_0:\mcB_0\rightarrow \mcN, g_0:\mcC_0\rightarrow \mcN$ such that $tp^\mcN(f_0(A_0)) \neq tp^\mcN(g_0(A_0))$. Let $\bar a_1',\bar a_2'$ be the subtuples of $\bar a_1$ and $\bar a_2$ which are represented in $A_0$. Both $\bar a_1$ and $\bar a_2$ had the same atomic diagram in $\overline \mcN$ thus $tp^\mcN(\bar a_1')=tp^\mcN(\bar a_2')$. As $\mcB_0$ and $\mcC_0$ contain witnesses for the isolating formulas of $tp^\mcN(\bar a_1)$ and $tp^\mcN(\bar a_2)$ respectively these witnesses also isolate $tp(\bar a_1')$  and $tp(\bar a_2')$. Thus we conclude that $tp^\mcN(f_0(A_0))=tp^\mcN(g_0(A_0))$ has to hold, which is a contradiction to what we previously showed.
\end{proof}
If we do not have the amalgamation property in the age of a $\Sigma_1-$homo\-ge\-ni\-zable structure, then for each diagram $f:\mcA\to\mcB, g:\mcA\to\mcC$ which does not have an amalgam there should be a homogenizing formula such that for some tuple $\bar a\in A$, this tuple satisfies the homogenizing formula in $\mcB$ but does not satisfy this formula in $\mcC$. This is the core reasoning behind the following lemma.
\begin{lma}
If $\mcM$ is a homogenizable model-complete structure then $\mcM$ is $\Sigma_1-$homogenizable, with all types isolated by a conjunction of the homogenizing formulas and quantifier free formulas, and $Age(\mcM)$ satisfies SEAP.
\end{lma}
\begin{proof}
Model-completeness is equivalent to the condition that each formula is equivalent to a $\Sigma_1-$formula, thus we may assume that the homogenizing formulas are $\Sigma_1-$formulas. The type of a tuple may then be isolated by a conjunction of homogenizing formulas and quantifier free formulas, since the structure is homogenizable. \\\indent 
In order to prove that $Age(\mcM)$ satisfies SEAP assume that $Age(\mcM)$ does not satisfy AP. We will show that SEAP$_{k,m}$ is satisfied where $k$ is the maximum among the number of free variables among homogenizing formulas and $m$ is the maximum among the number of bound variables among the homogenizing formulas. Assume $\mcA,\mcB,\mcC\subseteq \mcM$ with embeddings $f:\mcA\rightarrow \mcB$, $g:\mcA\rightarrow \mcC$ be without an amalgam. We conclude that $tp^\mcM(f(A))\neq tp^\mcM(g(A))$, however since they have the same atomic diagram there have to exist homogenizing formulas $\exists \bar y\varphi(\bar x,\bar y),\exists \bar y \psi(\bar x,\bar y)$, where $\varphi$ and $\psi$ are quantifier free, such that for some $\bar a_0\in A$, $\mcM\models \exists \bar y\varphi (f(\bar a_0),\bar y)\wedge \neg \exists \bar y\varphi(g(\bar a_0),\bar y) \wedge \neg \exists \bar y \psi(f(\bar a_0),\bar y)\wedge \exists \bar y \psi(g(\bar a_0),\bar y)$. Note that we may assume
\begin{equation}\label{eq1}\mcM\models \forall \bar x \Big(\big(\exists \bar y\varphi(\bar x, \bar y) \rightarrow \neg \exists \bar y \psi(\bar x,\bar y)\big)\wedge \big (\exists \bar y\psi(\bar x, \bar y) \rightarrow \neg \exists \bar y \varphi(\bar x,\bar y) \big)\Big ).
\end{equation}
Let $\mcA_0 = \mcA\reduct {\bar a_0}$, let $\mcB_0$ be $\mcB$ extended with a tuple witnessing the existential quantifier in $\exists \bar y\varphi(f(\bar a_0),\bar y)$ and let $\mcC_0$ be $\mcC$ extended with a tuple witnessing the existential quantifier in $\exists \bar y\varphi(g(\bar a_0),\bar y)$. Then $f_0: \mcA_0\rightarrow \mcB_0$, $f_0 = f\reduct {A_0}$ and $g_0:\mcA_0\rightarrow \mcC_0$, $g_0 = g\reduct {A_0}$ can not have an amalgam since (\ref{eq1}) hold and $\mcM\models \exists \bar y\varphi(f_0(\bar a_0),\bar y) \wedge \exists \bar y \psi(g_0(\bar a_0),\bar y)$.
\end{proof}
Combining the previous two lemmas we now have a proof for Theorem \ref{mcmainthm}.
In \cite{C} Covington asks whether all homogenizable classes have a homogenizable model companion. In the notation of \cite{C}, the previous theorem implies that we can find a homogenizable class without a homogenizable model companion if and only if there is a homogenizable class not satisfying SEAP. The author does not know whether such a class exists and hence the question remains open. 

%% file: sym.tex
In this section we characterize the boundedly homogenizable structures. We try to find out whether all model-complete homogenizable structures are boundedly homogenizable, but only find that this is the case for homogenizable structures with certain model theoretic properties. The following proposition give us a good understanding of the basic properties of boundedly homogenizable structures.
\begin{prp}\label{symbasicprp}
If $\mcM$ is a homogenizable countably infinite structure then the following are equivalent. 
\begin{enumerate}[label=(\roman*)]
\item $\mcM$ is a boundedly homogenizable structure.
\item For each finite $\mcA\subseteq \mcM$ there is a finite $\mcB$ with $\mcA\subseteq\mcB\subseteq\mcM$ such that each embedding $f:\mcB\rightarrow \mcM$ may be extended to an automorphism.
\item $\mcM$ is model-complete and for each $\mcA\subseteq \mcM$ there is an amalgamation base $\mcB$ for $Age(\mcM)$ such that $\mcA\subseteq\mcB\subseteq\mcM$.
\end{enumerate}
\end{prp}
\begin{proof}
$(i)$ and $(ii)$ are equivalent by Lemma \ref{qrfreelma} and the definition of being boundedly homogenizable. 
We prove $(iii)$ implies $(i)$ by Lemma \ref{amtofree} and to show that
$(i)$ implies $(iii)$ we use Lemma \ref{symtosigma} to get model-completeness and Lemma \ref{prpamfree} to get the amalgamation bases.
\end{proof}
As model-completeness is a very important property for homogenizable structures it is interesting to see that all boundedly homogenizable structures are model-complete.
\begin{lma}\label{symtosigma}
If a structure $\mcM$ is boundedly homogenizable then it is $\Sigma_1-$ ho\-mo\-geni\-zable and $Th(\mcM)$ is model-complete.
\end{lma}
\begin{proof}
Among the formulas which homogenize $\mcM$, assume that the largest number of free variables is $r$. Let $\bar a_1,\ldots,\bar a_n$ be realizations of all the different types on $1,\ldots,r-$tuples in $\mcM$. For each $i=1,\ldots,n$ let $\bar b_i\in M$ be such that $tp(\bar a_i \bar b_i)$ is isolated by a quantifier free formula and let $\chi_i$ be the atomic diagram of $\bar a_i\bar b_i$. It is clear that $\exists \bar x \chi_{i}(\bar y,\bar x)$ isolates $tp(\bar a_i)$. For each $i\in \{1,\ldots, n\}$, adding a relation symbol $R_i$ representing the formula $\exists \bar x\chi_i(\bar y, \bar x)$ will hence be a refinement of the homogenization, since it implies the old homogenizing formula. Hence this new homogenization is of the form $\Sigma_1$ and all types are isolated by a conjunction of $\Sigma_1-$formulas, thus the theory is model-complete.
\end{proof}
\noindent We now have the tools needed in order to prove Theorem \ref{symmainthm}.

\begin{proof}[Proof Theorem \ref{symmainthm}]
The existence of such a structure is clear since the \FR limit is homogeneous and hence boundedly homogenizable.
If $\mcM$ is boundedly homogenizable and $Age(\mcM)=\mbK$ Lemma \ref{symtosigma} then implies that $\mcM$ is model-complete, but Saracino \cite{S} has shown that there is always a unique model-complete countably infinite structure such that $Age(\mcM)=\mbK$, hence $\mcM$ must be this structure. Since every homogeneous structure is model-complete, $\mcM$ must be isomorphic to the \FR limit of $\mbK$.
\end{proof}
If $\mcM$ is not homogeneous yet homogenizable with an age satisfying the amalgamation property we have two choices for our favorite related model-complete structure. One choice is the Fra\"{i}ss\'{e} limit, which coincides with the model companion, however the second choice is the homogeneous structure which is gotten when adding new relational symbols. These two structures are not the same and do not even need to be reducts of one another. \\\indent The converse of Lemma \ref{symtosigma} can be formulated in the following question, to which the author does not know the answer.
\begin{quest} Does there exist a model-complete homogenizable structure which is not boundedly homogenizable?
\end{quest}

%% file: uniform.tex
The uniformly homogenizable structure have some tuple (or tuples) which determines the types of all other tuples in the structure. This notion makes us believe that if we have a $\Sigma_1-$homogenizable structure then it should be possible to witness the existential quantifiers of the homogenizing formulas for all tuples with a single uniform tuple.
\begin{defi}
A $\Sigma_1-$homogenizable structure $\mcM$ with homogenizing formulas $\exists \bar x\varphi_i(\bar y,\bar x)$ for $i=1,\ldots,n$ ($\varphi_i$ is quantifier free) has \textbf{uniformly homogenizing formulas} if for each $i = 1,\ldots,n$ 
\[\mcM\models \exists \bar x \forall \bar y \big(\exists \bar x_0 \varphi_i(\bar y,\bar x_0) \rightarrow \varphi_i(\bar y,\bar x) \big)\]
\end{defi}
We will prove that having uniformly homogenizing formulas is equivalent with being uniformly homogenizable, among some other characterizing properties in the spirit of Proposition \ref{symbasicprp}.
\begin{prp}\label{uniprp} If $\mcM$ is a homogenizable countably infinite structure then the following are equivalent:
\begin{enumerate}[label=(\roman*)]
\item$\mcM$ is uniformly homogenizable.
\item $\mcM$ has uniformly homogenizing formulas.
\item There is a finite structure $\mcN\subseteq\mcM$ such that for each finite structure $\mcA$ such that $\mcN\subseteq \mcA\subseteq\mcM$ and embedding $f:\mcA\rightarrow\mcM$, $f$ may be extended into an automorphism.
\item $\mcM$ is model-complete and there exists a finite structure $\mcN\subseteq \mcM$ such that each finite $\mcA\subseteq \mcM$ such that $\mcN$ is embeddable in $\mcA$ is an amalgamation base.
\end{enumerate}
\end{prp}
\begin{proof}
$(i)$ is equivalent to $(ii)$ is shown in Lemmas \ref{unilma1} and \ref{unilma2}. To show that $(i)$ is equivalent to $(iii)$ we use Lemma \ref{qrfreelma}. The uniformly homogenizable structures are boundedly homogenizable so $(i)$ implies $(iv)$ follows from Lemma \ref{symtosigma} and Lemma \ref{prpamfree}. By Lemma \ref{amtofree} the converse follows.
\end{proof}
We prove that a structure having uniformly homogenizing formulas implies that the structure is uniformly homogenizable by collecting the uniform witnesses for the homogenizing formulas together, and then show that these actually form a tuple whose type, and its extensions, are isolated by quantifier free formulas. 
\begin{lma}\label{unilma1} If $\mcM$ is homogenizable with uniformly homogenizing formulas then $\mcM$ is uniformly homogenizable.
\end{lma}
\begin{proof}
Assume that $\exists \bar x\varphi_1(\bar y,\bar x),\ldots,\exists \bar x\varphi_n(\bar y,\bar x)$ are the $\Sigma_1-$homogenizing formulas. Since we assume that these formulas are uniformly homogenizing there exist tuples $\bar a_1,\ldots,\bar a_n$ such that for each $i=1,\ldots,n$
\begin{equation}\label{eqlma43}
\mcM\models \forall \bar y \big(\exists \bar x_0 \varphi_i(\bar y,\bar x_0) \rightarrow \varphi_i(\bar y,\bar a_i) \big).
\end{equation}
We will now show that $\bar a = \bar a_1\ldots\bar a_n$ is a tuple witnessing that $\mcM$ is uniformly homogenizable. 
For any $\bar b\in M$, we will show that $tp(\bar a\bar b)$ is isolated by a quantifier free formula we will do downwards induction on the number of subtuples of $\bar a\bar b$ which satisfy some homogenizing formulas.\\\indent
As a base case of the induction assume that for any $\bar b'\in M$ such that $\mcM\reduct \bar b \cong \mcM\reduct \bar b'$, $\bar b$ has the highest number of subtuples satisfying the formulas in $\{\exists \bar x\varphi_i(\bar y,\bar x)\}_{i\in [n]}$.
As equation (\ref{eqlma43}) hold for the subtuples of $\bar a$, for each tuple $\bar c\bar d$ such that there is an isomorphism $f:\mcM\reduct \bar a\bar b \rightarrow \mcM\reduct \bar c\bar d$ and for any subtuple $\bar e_0$ of $\bar a\bar b$ and $i=1,\ldots,n$ if $\mcM\models \exists \bar x\varphi_i(\bar e_0, \bar x)$ then $\mcM\models\exists \bar x\varphi_i(f(\bar e_0),\bar x)$. However the maximality of $\bar b$ proves that this implication is an equivalence, thus $\mcM\models \exists \bar x\varphi_i(\bar e_0, \bar x)$ iff $\mcM\models\exists \bar x\varphi_i(f(\bar e_0),\bar x)$. As $\bar a \bar b$ and $\bar c \bar d$ satisfy the same atomic diagram and homogenizing formulas on respective subtuples, it is clear that $tp(\bar a \bar b) = tp(\bar c \bar d)$ hence the type is isolated by its atomic diagram.\\\indent 
As induction hypothesis we have that for each tuple $\bar c_0\bar d_0$ such that there is an isomorphism $f:\mcM\reduct \bar a\bar b\rightarrow \mcM\reduct \bar c_0\bar d_0$, if $\bar c_0\bar d_0$ has more subtuples satisfying $\exists x \varphi_i(\bar y,\bar x)$ for $i=1,\ldots,n$ then it has quantifier free isolation and hence we can not have an isomorphism to $\mcM\reduct \bar a\bar b$, as $tp(\bar a\bar b)\neq tp(\bar c_0\bar d_0)$. If on the other hand $\bar c_0\bar d_0$ have the same amount of subtuples satisfying homogenizing formulas, the same reasoning as previously in this proof (when we had maximal amount of subtuples) implies that $tp(\bar c_0\bar d_0)=tp(\bar a\bar b)$. Since $\bar c_0\bar d_0$ is an arbitrary tuple with the same atomic diagram as $\bar a\bar b$ we conclude that the type $tp(\bar a\bar b)$ is isolated by a quantifier free formula.
\end{proof}
To get the converse of the previous lemma we may need to change the homogenizing formulas so that they depend on tuples inducing types isolated by quantifier free formulas, and then show that the newly created formulas are uniformly homogenizing formulas.
\begin{lma}\label{unilma2} If $\mcM$ is a uniformly homogenizable structure, then $\mcM$ may be homogenized using only uniformly homogenizing formulas.
\end{lma}

\begin{proof}
Let $\bar a \in M$ be such that for each $\bar b \in M$, $tp(\bar a\bar b)$ is isolated by a quantifier free formula. Assume that the highest arity among homogenizing formulas is $r$.
For any $k\in [r]$, let $\bar b_1,\ldots,\bar b_m \in M$ be an as large set as possible of $k-$tuples such that $tp(\bar b_1) = \ldots=tp(\bar b_m)$ yet if $\mcB^k_i = \mcM\reduct\bar a\bar b_i$ then $\mcB^k_{i_0}\not\cong\mcB^k_{i_1}$ for any distinct $i_0,i_1 \in [m]$. Let $\chi_{k,i}$ be the atomic diagram of $\mcB^k_{i}$. Since the types of the tuples $\bar b_1,\ldots,\bar b_m$ are the same $\mcM\models \exists \bar x \chi_{k,i}(\bar x,\bar b_j)$ for each $i, j\in[m]$. Since the atomic diagram of $\mcB_i^k$ isolates $tp(\bar b_i)$ the disjunction $\bigvee_{i=1}^m\exists \bar x \chi_{k,i}(\bar x,\bar y)$ thus isolates $tp(\bar b_1)$. Note that $\bigvee_{i=1}^m\exists \bar x \chi_{k,i}(\bar x,\bar y)$ is equivalent to $\exists \bar x\bigvee_{i=1}^m \chi_{k,i}(\bar x, \bar y)$. Thus we may in this way create, for each $k\in[r]$ and $k-$type $p$, a $\Sigma_1-$formula which isolates $p$ and whose existential quantifier is witnessed by $\bar a$. As all homogenizing formulas have arity at most $r$, these new formulas will work as homogenizing formulas for $\mcM$ and they are clearly uniformly homogenizing.
\end{proof}
\noindent As Proposition \ref{uniprp} is now proven, we will finish this section with some results showing both how important the uniformly homogenizable structures are for the homogenizable structures, but also how trivial they might be. In the following proposition we work with $\mcM^{eq}$. This structure is obtained from $\mcM$ by adding a new element for each equivalence class of each $\emptyset-$definable equivalence relation on each power $M^n$ of $\mcM$ and expanding the language correspondingly, to indicate which tuples lie in which equivalence classes. This construction is very useful especially in the classification theory part of model theory, but as we will not use it in further detail we refer the reader to Chapter 4.3 in \cite{H} for a complete definition. Note that a \textbf{finite expansion} $\mcN$ of the $V-$structure $\mcM$ is a $V'-$structure of a finite vocabulary $V'\supseteq V$ such that $M\subseteq N\reduct V$ and $|N|-|M|$ is finite.
\begin{prp}\label{mequni}
For each homogenizable structure $\mcM$ there exists a finite expansion $\mcN\subseteq \mcM^{eq}$ such that $\mcN$ is uniformly homogenizable.
\end{prp}
\begin{proof}
Let the homogenizing formulas of $\mcM$ be $\varphi_1(\bar x_1),\ldots,\varphi_n(\bar x_n)$. These are by definition without parameters, and hence the formula 
\[\xi_i(\bar x,\bar y)\hspace{1cm} \Leftrightarrow \hspace{1cm} \varphi_i(\bar x)\leftrightarrow \varphi_i(\bar y)\]
defines an equivalence relation. Let $V' = V \cup \{P_i(y): i=1,\ldots n\}  \cup \{R_i(y,\bar x_i)\: i = 1,\ldots,n\} \subseteq V^{eq}$ and $\mcN'\subseteq \mcM^{eq}$ be such that $\mcN'$ contains all of $M$ and only the equivalence classes of all the formulas $\xi_i$. Note that for each $i$, $P_i$ in $\mcM^{eq}$ is the relation which holds for elements representing equivalence classes for $\xi_i$ and $R_i$ relates equivalence classes of $\xi_i$ to tuples in that equivalence class. Let $\mcN = \mcN'\reduct V'$, it is now easy to show that this structure is uniformly homogenizable with the uniform witness being the tuple containing all $2n$ elements representing the equivalence classes.
\end{proof}
Algebraic formulas are formulas which are only satisfied by a finite number of tuples. 
If we want an easy example of a homogenizable structure we may take any homogeneous structure and add a finite number of elements which are $\emptyset-$definable and with the same atomic diagram as something in the rest of the structure, but with a different type.
The following proposition ensures that any such structure will be uniformly homogenizable. It is interesting to compare the assumptions of the proposition with Example \ref{nonsymexa} which is both $\Sigma_1-$homo\-ge\-ni\-zable and homogenizable using only algebraic formulas, yet we may not find a homogenization of the structure which satisfies both of these properties at the same time.
\begin{prp}\label{finsatprp}
If $\mcM$ is $\Sigma_1-$homogenizable such that the homogenizing formulas are algebraic then $\mcM$ is uniformly homogenizable.
\end{prp}
\begin{proof}
Let $\exists \bar x_1\varphi_1(\bar x_1,\bar y),\ldots,\exists \bar x_n\varphi_n(\bar x_n,\bar y)$ be the homogenizing algebraic formulas, and assume that $\bar a_1,\ldots,\bar a_m$ are the tuples satisfying these formulas with existential quantifiers witnessed by $\bar b_1,\ldots,\bar b_m$ respectively. Let $\bar b=\bar b_1\ldots\bar b_m$ and let $\bar x=\bar x_1\ldots\bar x_m$ be a variable tuple of the same length. For each formula $\varphi_i(\bar x_i,\bar y)$ create a formula $\varphi_i'(\bar x,\bar y)$ which is equivalent with 
\[\varphi_i(\bar x_i,\bar y)\wedge\bigwedge_{j\neq i}\bar x_j=\bar x_j.\] It is clear that $\exists \bar x\varphi'_1,\ldots,\exists \bar x\varphi_n'$ also work as homogenizing formulas, and the element $\bar b$ can be chosen to witness $\bar x$ in all of the formulas. It follows that $\exists \bar x\varphi'_1,\ldots,\exists \bar x\varphi_n'$ are uniformly homogenizing formulas for $\mcM$ and thus $\mcM$ is uniformly homogenizable by Proposition \ref{uniprp}.
\end{proof}

%% file: binclassi.tex
In section \ref{prelim} we defined the unavoidably homogenizable structures. However nowhere in the definition of unavoidably homogenizable structures do we demand that such a structure has to be homogenizable or even $\omega-$categorical. This follows though from the very tight restriction we keep on the complete types.
\begin{lma}\label{unavoidlma}
If $\mcM$ is unavoidably homogenizable, then $\mcM$ is $\Sigma_1-$homo\-ge\-ni\-zable. 
\end{lma}
\begin{proof}
Assume $k\in\mbbN$ is such that $\mcM$ is $k-$unavoidably homogenizable and let $\bar a_1,\ldots,\bar a_n\in M^k$ be such that all different atomic diagrams are represented. Note that this is finite since the vocabulary is finite relational and it thus becomes clear that $\mcM$ is $\omega-$categorical. Let $\chi_i$ be the atomic diagram of the tuple $\bar a_i$. It is now clear that all the formulas of the form $\exists \bar x\chi_i(\bar y,\bar x)$ together form $\Sigma_1-$homogenizing formulas, as each tuple of size less than $k$ has its type isolated by such a formula.
\end{proof}
As the properties of unavoidably homogenizable structures are very close to the uniform and boundedly homogenizable structures, we may prove a proposition which is similar to Proposition \ref{symbasicprp}.
\begin{prp}\label{unavprp} Assume that $\mcM$ is an $\omega-$categorical countably infinite structure and $k\in\mbbN$, then the following are equivalent.
\begin{enumerate}[label=(\roman*)]
\item $\mcM$ is $k-$unavoidably homogenizable.
\item For each $\mcA\subseteq \mcM$ with $|A|\geq k$ and each embedding $f:\mcA\rightarrow\mcM$, $f$ may be extended into an automorphism.
\item $\mcM$ is model-complete and each finite $\mcA\subseteq \mcM$ such that $|A|\geq k$ is an amalgamation base for $Age(\mcM)$.
\end{enumerate}
\end{prp}
\begin{proof}
$(i)$ is equivalent to $(ii)$ follows from Lemma \ref{qrfreelma}. If we assume $(i)$, Lemma \ref{unavoidlma} implies that $\mcM$ is homogenizable and thus the definition of unavoidably homogenizable implies that $\mcM$ is boundedly homogenizable. Thus Lemma \ref{symtosigma} and Lemma \ref{prpamfree} implies $(iii)$. That $(iii)$ implies $(i)$ follows from Lemma \ref{amtofree}.
\end{proof}

\begin{rmk} The author classified the unavoidably homogenizable graphs in \cite{A}. 
However we have no real hope of classifying the unavoidably homogenizable structures properly without first classifying the homogeneous structures, since from homogeneous structures we may easily create similar unavoidably homogenizable structures in the following way. Let $\mcM_0,\mcM_1$ be two homogeneous structures over a vocabulary $V$ and let $R_0,R_1$ be $k-$ary relational symbols which are not in $V$. Let $\mcN$ be the structure over $V\cup\{R_0,R_1\}$ with universe $M_0\dot\cup M_1$ such that for $i\in\{0,1\}$ $(\mcN\reduct M_i)\reduct V = \mcM_i$ and no relations from $V$ hold between elements in $M_0$ and $M_1$ in $\mcN$. Furthermore, create $\mcN$ such that $\mcN\models R_i(\bar a)$ for every $k-$tuple $\bar a\in M_i$ of distinct elements. \\\indent
The structure $\mcN$ is unavoidably homogenizable since for any $\mcA\subseteq N$ such that $|A|\geq 2k-1$ there will be a tuple which satisfies $R_0$ or $R_1$, but then if $\mcA$ is embedded in $\mcN$ the parts belonging to $M_0$ and $M_1$ have to be mapped to the correct side, and since $\mcM_0$ and $\mcM_1$ are homogeneous, this may be extended to an automorphism. This proves that $\mcN$ is unavoidably homogenizable by Proposition \ref{unavprp}.
\end{rmk}
It seems that we may at least assume that all elements are of the same atomic diagram in an unavoidably homogenizable structure, as the following proposition shows.
\begin{prp}
Let $\xi(x,y)$ be the equivalence relation which holds if two elements satisfy the same atomic diagram.  If $\mcM$ is $k-$unavoidably homogenizable then each infinite equivalence class $A$ of $\xi$ is such that $\mcM\reduct A$ is a $k-$unavoidably homogenizable structure.
\end{prp}

\begin{proof}
Let $\mcA = \mcM\reduct A$ and choose $B\subseteq A$ such that $|B|\geq k$. If $f:\mcA\reduct B\rightarrow \mcA$ is an embedding then it is also an embedding into $\mcM$, and hence by Proposition \ref{unavprp} there is an automorphism $g$ of $\mcM$ extending $f$. However the elements in $A$ are exactly those who have the same atomic diagram, hence $g$ must map $A$ to $A$, so $g\reduct A$ is an automorphism of $\mcA$ which extends the embedding $f$, so again by Proposition \ref{unavprp}, it follows that $\mcA$ is $k-$unavoidably homogenizable.
\end{proof}

%% file: one.tex
\noindent The structures which we may homogenize by only adding new unary relational symbols are quite special and we call these structures unary homogenizable. We quickly see that, unless it is homogeneous, such a structure is non-transitive i.e. there are elements $a, b$ such that $a$ can not be mapped to $b$ by an automorphism. In this section we explore these structures further, exposing a quite close relation between unary homogenizable and uniformly homogenizable structures in Theorem \ref{unatouni}.\\\indent In a structure $\mcM$, the algebraic closure of a set $X\subseteq M$ is the set of all elements $a\in M$ such that $tp(a/ X)$ is only realized by a finite number of elements in $M$. If for each $X\subseteq M$ the algebraic closure of $X$ equals to $X$ then we say that the algebraic closure is degenerate.
Any homogeneous structure $\mcM$ such that $Age(\mcM)$ satisfies the disjoint amalgamation property has degenerate algebraic closure, so the restriction in the following theorem is not as large as it might seem. 
Note that if $\mcM_1 = (D_1; R_1^{\mcM_1},\ldots,R_n^{\mcM_1})$, $\mcM_2 = (D_2 ; R_1^{\mcM_2},\ldots,R_n^{\mcM_2})$ are structures of the same signature then we define the union structure in the following way $\mcM_1\cup\mcM_2 = (D_1\cup D_2 ; R_1^{\mcM_1}\cup R_1^{\mcM_2},\ldots,R_n^{\mcM_1}\cup R_n^{\mcM_2})$.
\begin{thm}\label{unatouni}
If $\mcM$ is a countably infinite unary boundedly homogenizable structure with degenerate algebraic closure then there are infinite uniformly homogenizable structures $\{\mcN_i\}_{i\in I}$ with only finitely many different isomorphism types such that
\[\mcM = \bigcup_{i\in I}\mcN_i\]
\end{thm}  
As a first step to prove the above theorem we show the following Lemma.
\begin{lma}\label{unprp1}
Let $\mcM$ be unary homogenizable and $\bar a,\bar b\in M$. If both $tp(\bar a)$ and $tp(\bar b)$ are isolated by quantifier free formulas then $tp(\bar a\bar b)$ is isolated by quantifier free formulas.
\end{lma}
\begin{proof}
Assume that $\bar c\bar d\in M$ are such that there is an isomorphism $f:\mcM\reduct \bar a\bar b \rightarrow \mcM\reduct \bar c\bar d$ and $\varphi(x)$ is a homogenizing formula such that for some $a_0\in \bar a\bar b, \mcM\models \varphi(a_0)$. Either $a_0\in \bar a$ or $a_0\in\bar b$ and since both of the tuples have types isolated by quantifier free formulas $\mcM\models \varphi(a_0)$ if and only if $\mcM\models\varphi(f(a_0))$.
If we homogenize $\mcM$ we add relations for $\varphi$ on the same elements in $\bar a\bar b$ and $\bar c\bar d$, i.e. $f$ will still be an isomorphism, when extended to the new vocabulary. Thus $\bar a\bar b$ and $\bar c\bar d$ satisfy the same homogenizing formulas, hence $f$ may be extended to an automorphism and hence $tp(\bar a\bar b)=tp(\bar c\bar d)$.
\end{proof}
\noindent The type condition in the previous lemma does not imply unary homogenizability, however we can at least show that $\mcM$ must have at least one unary homogenizing formula. 
\begin{cor}
Let $\mcM$ be a non-homogeneous homogenizable $V-$structure such that for any $\bar a,\bar b \in M$ with $tp(\bar a)$ and $tp(\bar b)$ isolated by quantifier free formulas, $tp(\bar a\bar b)$ is isolated by a quantifier free formula. Then there is a vocabulary $V'\supseteq V$ and a $V'-$structure $\mcN$ which is non-homogeneous and unary homogenizable such that $\mcN\reduct V = \mcM$ and $Aut(\mcM)=Aut(\mcN)$.
\end{cor}
\begin{proof}
If no unary homogenizing formulas exist, then for each $a,b\in M$, $tp(a)$ and $tp(b)$ are isolated by quantifier free formulas. Thus by the assumption $tp(ab)$ is isolated by a quantifier free formula. It follows by induction that for any tuple $\bar c\in M$, $tp(\bar c)$ is isolated by a quantifier free formula hence $\mcM$ has quantifier elimination which is equivalent with being homogeneous, a contradiction.
\end{proof}

\noindent Another corollary from the previous Lemma shows that the boundedly homogenizable structures are quite easy to reach from the unary homogenizable.
\begin{cor}
If $\mcM$ is a unary homogenizable structure such that for each $a\in M$ there is $\bar b\in M$ such that $tp(a\bar b)$ is isolated by a quantifier free formula then $\mcM$ is boundedly homogenizable.
\end{cor}
\begin{proof}
If $\bar c = (c_1,\ldots,c_n) \in M$ let $\bar b_1,\ldots,\bar b_n\in M$ be such that for each $i=1,\ldots,n$, $tp(c_i\bar b_i)$ is isolated by a quantifier free formula. Lemma \ref{unprp1} now gives us (through an obvious use of induction on $n$) that $tp(c_1\ldots c_n\bar b_1\ldots \bar b_n) = tp(\bar c\bar b_1\ldots \bar b_n)$ is isolated by a quantifier free formula. 
\end{proof}
\noindent We continue towards the goal of proving Theorem \ref{unatouni} by introducing a substructure consisting of all elements which behave nicely with respect to a certain tuple.
\begin{defi} Let $\mcM$ be a structure with $\bar a\in M$ and define the set
\[X_{\bar a}= \{b\in M : tp(b\bar a) \text{ is isolated by a quantifier free formula}\}.\]
Define the structure $\mcN_{\bar a} = \mcM\reduct X_{\bar a}\cup \bar a$
\end{defi}
\noindent The structure $\mcN_{\bar a}$ is focused around $\bar a$ and indeed this tuple is so special that it becomes the element witnessing that $\mcN_{\bar a}$ is uniformly homogenizable.
\begin{lma}\label{unlma2} If $\mcM$ is a countably infinite unary homogenizable structure with $\bar a\in M$, then the following hold:
\begin{itemize}
\item If $\bar b\in X_{\bar a}$ then $X_{\bar a}\subseteq X_{\bar a\bar b}$.
\item If $\bar b\in M$ and $tp^\mcM(\bar a) = tp^\mcM(\bar b)$ then $\mcN_{\bar a} \cong \mcN_{\bar b}$
\item $\mcN_{\bar a}$ is uniformly homogenizable.
\end{itemize}
\end{lma}
\begin{proof}
In order to prove the first statement we may assume without loss of generality that the tuple $\bar b\in X_{\bar a}$ consists of a single element $b$. If $c\in X_{\bar a}$ then, since $b\in X_{\bar a}$, Lemma \ref{unprp1} implies that $tp(cb\bar a)$ is isolated by a quantifier free formula, and hence $c \in X_{b\bar a}$.\\\indent
For the second part, assume that $\bar b\in M$ and $tp^\mcM(\bar a)= tp^\mcM(\bar b)$. This implies that there is an automorphism of $\mcM$ mapping $\bar a$ to $\bar b$. The restriction of this automorphism to $X_{\bar a}$ is then an isomorphism between $\mcN_{\bar a}$ and $\mcN_{\bar b}$.\\\indent
 For the third part first note that if $X_{\bar a}$ is finite, then the structure $\mcN_{\bar a}$ is uniformly homogenizable, by taking as uniform witness the whole structure, hence we assume $X_{\bar a}$ is infinite. Choose any $\alpha\in X_{\bar a}$, we will show that $\bar a \alpha$ is a witness for the uniform homogenization. Assume that for some $\bar b,\bar c\in X_{\bar a}$ there is an isomorphism $f:\mcN_{\bar a}\reduct\bar b\bar a\alpha \rightarrow \mcN_{\bar a}\reduct\bar c\bar a\alpha$. By Lemma \ref{unprp1} $tp^\mcM(\bar b\bar a\alpha)$ is isolated by a quantifier free formula thus $tp^\mcM(\bar b\bar a\alpha) = tp^\mcM(\bar c\bar a \alpha)$. However as $\mcM$ is saturated, this means that $f$ may be extended into an automorphism of $\mcM$. The restriction of this automorphism to $\mcN_{\bar a}$ implies that $tp^{\mcN_{\bar a}}(\bar b\bar a\alpha) = tp^{\mcN_{\bar a}}(\bar c\bar a \alpha)$.
\end{proof}

\begin{proof}[Proof of Theorem \ref{unatouni}]
Assume that the highest arity among relational symbols in $V$ equals to $\rho$ and let $\{\bar a_i\}_{i\in I}$ enumerate all $\rho-$tuples for some index set $I$. If there is a tuple $\bar a_i$ such that for each tuple $\bar b\in M$, $tp(\bar a_i\bar b)$ is isolated by a quantifier free formula, then $\mcM$ is uniformly homogenizable, and hence we are trivially done. Without loss of generality, we may thus assume that each tuple in $\{\bar a_i\}_{i\in I}$ does not have a type isolated by quantifier free formulas, since if $\bar a_i$ would be isolated by a quantifier free formula we can extend it to a tuple which is not hence all $\rho-$tuples are accounted for.
Since $\mcM$ is boundedly homogenizable, for each $\bar a_i$ let $\bar b_i$ be a tuple such that there is an element $c$ such that $tp(\bar a_i\bar b_ic)$ is isolated by a quantifier free formula. But the algebraic closure being degenerate implies that there is an infinite number of such elements $c$ hence $\mcN_{\bar a_i\bar b_i}$ is an infinite uniformly homogenizable structure by Lemma \ref{unlma2}. As $\mcM$ is $\omega-$categorical there are only a finite amount of different types of tuples $\bar a_i\bar b_i$. Hence by Lemma \ref{unlma2} there are only a finite number of isomorphism classes on $\{\mcN_{\bar a_i\bar b_i}\}_{i\in I}$. Since each $\rho-$tuple is contained in at least one of the structures we get $\mcM = \bigcup_{i\in I} \mcN_{\bar a_i\bar b_i}$.\end{proof}
The random bipartite graph is a non-unary homogenizable structure which is a union of uniformly homogenizable infinite structure. This follows as we may for each element $a$ let $\mcN(a)$ be the structure consisting of $a$ and all elements adjacent to $a$, thus no more edges than those to $a$ exist in $\mcN(a)$ and it is hence clear that $\mcN(a)$ is uniformly homogenizable (even unavoidably homogenizable by \cite{A}). The random bipartite graph is the union of all such structures $\mcN(a)$ for all elements $a$ and by the properties of the random bipartite graph all $\mcN(a)\cong\mcN(b)$ for all elements $a$ and $b$. \\\indent However this property does not hold for all boundedly homogenizable structures as we can see in examples such as \ref{treeexa}.

%% file: art.bbl
\begin{thebibliography}{9}
\bibitem{A} O. Ahlman, \textit{$\geq$k$-$homogeneous infinite graphs}, Preprint.
\bibitem{AK2} O. Ahlman, V. Koponen, \textit{Limit laws and automorphism groups of random non-rigid structures}, Journal of Logic and Analysis 7:2 (2015) 1-53.
\bibitem{AK} O. Ahlman, V. Koponen, \textit{Random l-colourable structures with a pregeometry}, arXiv 1207.4936, Accepted for publication in Mathematical Logic Quarterly.
\bibitem{BBPP} M. Bodirsky, D. Bradley-Williams, M. Pinsker, A. Pongracz, \textit{The universal homogeneous binary tree}, arXiv: 1409:2170
\bibitem{Ch} G. Cherlin, \textit{Relational complextiy of a finite permutation group}, preprint.
\bibitem{C} J. Covington, \textit{Homogenizable relational structures} Illinos J. Math. 34 (1990), no. 4, 731-743.
\bibitem{C2} J. Covington, \textit{A universal structure for N-free graphs}, Proc. London Math. Soc. Vol 58 (1989) 1-16.
\bibitem{D}M. Droste, \textit{Partially Ordered Sets with Transitive Automorphism Group}, Mem. Amer. Math. Soc., vol. 334, Amer. Math. Soc., Providence,
RI, 1985.
\bibitem{EF} H-D. Ebbinghaus, J. Flum, \textit{Finite model theory}, Springer verlag (2000).
\bibitem{Fr} R. \FR, \textit{Sur certaines relations qui g\'en\'eralisent lorder des nombres rationnels}, C. R. Acad. Sci. Paris 237 (1953) 540-542.
\bibitem{HHN} D. Hartman, J. Hubi$\check{\text{c}}$ka, J. Ne$\check{\text{s}}$et$\check{\text{r}}$il, \textit{Complexities of relational structures}, Math. Slovaca 65 (2015), no. 2, 229-246.
\bibitem{H} W. Hodges, \textit{Model theory}, Cambridge University Press (1993).
\bibitem{KR} D.J. Kleitman, B.L. Rothschild, \textit{Asymptotic enumeration of partial orders on a finite set}, Trans. Amer. Math. Soc. 205 (1975) 205-220.
\bibitem{KPR} P.G. Kolaitis, H.J. Prömel, B.L. Rothschild, \textit{$K_{l+1}$-free graphs: asymptotic structure and a 0-1 law}, Trans. Amer. Math. Soc. 303 (1987) 637-671.
\bibitem{K} V. Koponen, \textit{Asymptotic probabilities of extension properties and random l-colourable structures}, Annals of pure and applied logic, Vol 163 (2012) 391-438.
\bibitem{M} D. Macpherson, \textit{A survey of homogeneous structures}, Discrete mathematics Vol 311 (2011) 1599-1634.
\bibitem{S} D. Saracino, \textit{Model companions for $\aleph_0-$categorical theories}, Proc. Amer. Math. Soc. Vol 39 (1973) 591-598.
\bibitem{TZ} K. Tent, M. Ziegler, \textit{A course in model theory}, Cambridge university press (2012).
\end{thebibliography}
